\documentclass[11pt,reqno]{amsart}

\usepackage{amsfonts,amssymb}
\usepackage{graphicx,tikz}
\usepackage{tikz-cd} 
\usepackage{float}
\usepackage{amsmath, amsthm, amsfonts}
\usepackage{hyperref}
\usepackage{enumitem}  
\usepackage[capitalise]{cleveref}
\usepackage{comment}
\usepackage{enumitem}
\usepackage{amsrefs}
\usepackage[mathscr]{euscript}
\usepackage{nicefrac}
\usepackage[official]{eurosym}
\usepackage{nomencl}
\usepackage{hyperref}
\usepackage{tabto}
\usepackage{amssymb}
\usepackage{fancyhdr}
\usepackage{amscd}
\usepackage{pdfpages}
\usepackage{graphicx}
\usepackage[english]{babel}
\usepackage{makecell}

\usepackage[margin=1.15in]{geometry}

\theoremstyle{plain}

\newtheorem*{theorem*}{Theorem}
\newtheorem{theorem}{Theorem}[section]

\newtheorem{lemma}[theorem]{Lemma}

\theoremstyle{definition}
\newtheorem{definition}[theorem]{\scshape{Definition}}

\newtheorem{notation}[theorem]{\scshape{Notation}}

\newtheorem{remark}[theorem]{\scshape{Remark}}

\makeatletter
\def\cleardoublepage{\clearpage\if@twoside \ifodd\c@page\else
	\hbox{}
	\thispagestyle{empty}
	\newpage
	\if@twocolumn\hbox{}\newpage\fi\fi\fi}
\makeatother

\def\v{\mathbf{v}}

\numberwithin{equation}{section}

\author[P.~Suhajda]{P\'{e}ter Suhajda}
 
 \address{P\'{e}ter Suhajda:
KTH Royal Institute of Technology, 114 28 Stockholm, Sweden \newline
Stockholm University, 114 19 Stockholm, Sweden}
\email{suhajda@kth.se}

\author[A. Thillaisundaram]{Anitha Thillaisundaram} 
\address{Anitha Thillaisundaram: Centre for Mathematical Sciences, Lund University,  223 62 Lund, Sweden}
\email{anitha.thillaisundaram@math.lu.se}

\date{\today}

	
\title{On the Frobenius number for three variables}

\begin{document}

\begin{abstract}
For positive integers $a$, $b$, and $c$ which have no common divisor, the Frobenius number of $a$, $b$ and $c$ is defined to be the largest integer that cannot be expressed as a linear combination of $a$, $b$ and $c$  with non-negative integer coefficients. In 2017, Tripathi gave an algorithmic formula for the Frobenius number in three variables, however there were some minor inconsistencies in the formula. In this paper, we  settle these inconsistencies. 
\end{abstract}

\maketitle

\section{Introduction}

 For an integer $n\ge 2$, the \emph{Frobenius Problem} in $n$ variables is to determine the largest positive integer that cannot be expressed as a non-negative integer  combination of $n$ given positive integers which have no common divisor. This largest positive integer is known as the \emph{Frobenius number}. More precisely, given $a_1,\ldots,a_n\in\mathbb{N}$ with $\text{gcd}(a_1,\ldots,a_n)=1$, the \emph{Frobenius number} $g(a_1,\ldots,a_n)$ is
 \[
 \max\mathbb{N}\backslash\{a_1x_1+\cdots+a_nx_n\mid x_1,\ldots,x_n\in\mathbb{N}\cup\{0\}\}.
 \]
 
 The two-variable case, i.e. $n=2$, was the origin of the problem, and this case was considered by Sylvester in 1884, and it was proved by W.\,J.~Curran Sharp that $g(a_1,a_2)=a_1a_2-a_1-a_2$; see~\cite{Sylvester}. Apart from considering the Frobenius Problem for higher values of $n$, related aspects are studied  nowadays, such as the  pseudo-Frobenius number (see for instance~\cite{Branco}) and also concerning algorithmic aspects of computing the Frobenius number (see for instance~\cite{Ramirez}).

 Regarding the case $n=3$, it was shown by Curtis~\cite{Curtis} that there is no closed polynomial formula for the Frobenius number for $n=3$, and likewise for higher values of $n$.  Despite this, formulae have been given for certain arithmetic sequences; see \cite{Tripathi} for an overview of these results. For the general case for $n=3$, formulae, of an algorithmic nature, were given by Tripathi in~\cite{Tripathi}, which involved a highly technical approach. However, there were some minor discrepancies in Tripathi's results, which we resolve here. Though the entire framework of~\cite{Tripathi} still applies, and only minor adjustments need to be made along the way, and to the final statement.
 
 More specifically, the first discrepancy in \cite{Tripathi} occurs in Lemma 5 of~\cite{Tripathi}, in the computation of consecutive local minima of a certain function. This Lemma 5 is used in several results building up to the main results. In most cases, the statements of these other results still hold, just adjustments need to be performed to the proofs. However the statement of part of the final main result, i.e. \cite[Thm.~5(b)]{Tripathi}, is unfortunately inconsistent, as seen by the computational data for the Frobenius number that the first author  obtained during the writing of his Bachelor thesis at Lund University~\cite{Peter}. A key aspect for adjusting the proof of the main result relies on Lemma~\ref{lem:difference}, which establishes the correct distance between terms of a certain sequence.
 
 In this paper, we provide all necessary adjustments and corrections; in particular, see Lemma~\ref{lem:5} for the correction of \cite[Lem.~5]{Tripathi} and Lemma~\ref{lem:difference} for the correction of the aforementioned distance. For conciseness and for the reader's convenience, we reproduce many of Tripathi's arguments and proofs here, though we often expand on them and provide extra details for clarity. Due to the large amount of notation required in the statement of the algorithmic formulae, we refer the reader to Section~\ref{sec:greater-than} for the corrected statement, i.e. Theorem~\ref{thm:5}. However, we include here a summary of the formulae. Table~\ref{tab:summary} assumes that all three variables are pairwise coprime. If this is not the case, we recall that one can use Johnson's formula; cf. Theorem~\ref{thm:Johnson} below.

\begin{table}[hb!]
    \centering
    \begin{tabular}{|c|c|}
    \hline
        \makecell{\textbf{Conditions}} & \textbf{Formula for $g(a,b,c) + a$} \\
        \hline
        \makecell[cl]{$\ell\leq k$} & $ab-b$ \\
        \hline
        \makecell[cl]{$\ell > k$ \& $br < cq$} & \makecell{$\begin{cases}
    b\big((\lambda+1)(a-\ell)+r-1\big) & \text{if }\lambda\ge\frac{c(q-1)-br}{b(a-\ell)+c}, \\
    b(a-\ell-1)+c(q-\lambda-1) & \text{if }\lambda\le\frac{c(q-1)-br}{b(a-\ell)+c}.
\end{cases}$}\\
    \hline
    \makecell[cl]{$\ell > k$ \& $br > cq$, \\ $\mu < \lfloor \frac{r}{u}\rfloor$} & \small{$\max\left\{b(r-\mu u-1),b(u-1)+c\big(\mu(q+1)+q(\lfloor\tfrac{(a-\ell-r)\mu}{r}\rfloor+1)\big)\right\}+cq\lfloor\tfrac{a-\ell-1}{r}\rfloor$} \\
    \hline
    \makecell[cl]{$\ell > k$ \& $br > cq$, \\ $\mu < \lfloor \frac{r}{u}\rfloor$ \& $\Lambda > \Delta$} & $\max\left\{b(r-\Delta(a-\ell-r)-1),b(a-\ell-r-1)+c\big(\Delta(q+1)+q\big)\right\}+cq$ \\
    \hline
   \makecell[cl]{$\ell > k$ \& $br > cq$, \\ $\mu < \lfloor \frac{r}{u}\rfloor$ \& $\Delta' > \Lambda'$} & $\max\left\{b(r-1),b\big((a-\ell-1) \bmod r\big) +cq\right\}+cq\lfloor\tfrac{a-\ell-1}{r}\rfloor.$\\
    \hline
    \makecell[cl]{$\ell > k$ \& $br > cq$, \\ $\mu > \lfloor \frac{r}{u}\rfloor$} & $\max\left\{b(\widehat{x}-1)+cy_{w},\, b(\widehat{x}+u-x_{\mu}-1)+cy_{\mu}\right\}+cq\lfloor\tfrac{a-\ell-1}{r}\rfloor.$\\
    \hline
    \end{tabular}
    \vspace{1mm}
    \caption{Summary table of all formulae for three variables $a<b<c$, all assumed to be pairwise coprime. We refer  in particular to Notation~\ref{not:k-l}, Notation~\ref{not:m}, Theorem~\ref{thm:3}, Notation~\ref{notation} and Theorem~\ref{thm:5}  for all notation.}
    \label{tab:summary}
\end{table}

The algorithmic formulae for the Frobenius number in three variables is split over two results in \cite{Tripathi}, namely \cite[Thm.~3]{Tripathi} and \cite[Thm.~5]{Tripathi}. Additionally, using a slightly different parallel method, Tripathi provides alternative algorithmic formulae in \cite[Thm.~4]{Tripathi} and \cite[Thm.~6]{Tripathi}, with the analogous inconsistency in \cite[Lem.~6]{Tripathi}. In this paper, we only deal explicitly with Theorems 3 and 5 of \cite{Tripathi},  as our approach can be correspondingly adjusted to give the necessary corrections to the proof of \cite[Thm.~4]{Tripathi}, to the proof of \cite[Thm.~6]{Tripathi} and to the statement of \cite[Thm.~6(b)]{Tripathi}. For completeness, we provide the corrected verion of \cite[Thm.~6(b)]{Tripathi} at the end of the paper.

\medskip

\noindent \textit{Organisation.} In Section~\ref{sec:preliminaries} we collect together some preliminary results, and give the corrected versions of \cite[Lem. 5 and 7]{Tripathi}. Then, following the approach in~\cite{Tripathi}, we split the main results into two parts, Sections~\ref{sec:less-than} and~\ref{sec:greater-than}, depending on a condition concerning the values of $a$, $b$ and $c$. In particular, we note in Section~\ref{sec:less-than} that the statement of \cite[Thm.~3]{Tripathi} still holds, and in Section~\ref{sec:greater-than} we see that \cite[Thm.~5(a)]{Tripathi} is correct, but that \cite[Thm.~5(b)]{Tripathi} requires some minor adjustment.

\section{Preliminary results}\label{sec:preliminaries}

For the entirety of the paper, let $c>b>a\ge 2$ be pairwise coprime integers, i.e. with $\text{gcd}(a,b)=\text{gcd}(a,c)=\text{gcd}(b,c)=1$. Note that we may make this stronger assumption that $a,b,c$ are pairwise coprime, instead of just $\text{gcd}(a,b,c)=1$, owing to the following result of Johnson~\cite{Johnson}:

\begin{theorem}\label{thm:Johnson}
   \cite[cf.~Thm.~2]{Johnson} 
   Let $a,b,c\ge 2$ with $\textup{gcd}(a,b,c)=1$.    If $\textup{gcd}(b,c)=d$ and $b=db'$, $c=dc'$,    then
   \[
   g(a,b,c)=d\cdot g(a,b',c')+a(d-1).
   \]
\end{theorem}

We recall the following notation from \cite{Tripathi}: 
\begin{notation}\label{not:k-l}
Let $c>b>a\ge 2$ be pairwise coprime integers. Define
\[
k:=\left\lfloor \frac{c}{b}\right\rfloor,\qquad \ell :=cb^{-1} \pmod a,
\]
where, in this paper, whenever we reduce a number $z$ modulo $n$ we always assume that $z\in\{0,1,\ldots,n-1\}$. 
\end{notation}

\begin{remark}\label{rmk:c-bl}
As it was shown in \cite[Lem.~2]{Tripathi} that if $\ell\le k$, then $g(a,b,c)=g(a,b)$, from now onwards we will always assume that $\ell>k$. 
\end{remark}

We first record a simple observation for later use. Recall that our standing assumption for the rest of the paper is that $c>b>a\ge 2$ are pairwise coprime.

\begin{lemma}\label{lem:c-bl}
     Suppose $\ell>k$. Then $c-b\ell<0$.
\end{lemma}

\begin{proof}
    Suppose for a contradiction that $c-b\ell\ge 0$. Certainly we cannot have $c-b\ell=0$, since then $\ell=k$. So suppose $c-b\ell>0$. Note that $c-b<bk\le c$, so we have $c>b\ell> bk>c-b$. However, there is no space strictly between $c-b$ and $c$ for two distinct multiples of $b$. Hence the result.
\end{proof}

Next we recall the following result from \cite{Tripathi}. First, for non-negative integers $x,y$, set $\mathbf{v}(x,y):=bx+cy$ and we refer to this as the $\mathbf{v}$-value of $(x,y)$. 

\begin{lemma}
    \cite[Lem.~3]{Tripathi} \label{lem3} Suppose $\ell>k$. Then
    \[
    g(a,b,c)=\max_{1\le x<a}\left\{\min_{0\le t<a} \mathbf{v}\big((x+(a-\ell)t) \bmod a, t\big)\right\}-a.
    \]
\end{lemma}

This leads to the following definition from \cite{Tripathi}:
\begin{definition}
    \cite[Def.~1]{Tripathi}  Let $x\in\{0,1,\ldots,a-1\}$. 
    We say that $\mathbf{v}(x,y)$ is a \emph{local minimum} if
    \begin{align*}
    \mathbf{v}(x,y)&\le \min\left\{\mathbf{v}\big((x-\ell)\bmod a, \,y+1\big), \mathbf{v}\big((x+\ell)\bmod a, \,y-1\big)\right\}  \quad\text{when }y\in\{1,\ldots,a-2\},\\
    \mathbf{v}(x,y)&\le \mathbf{v}\big((x-\ell)\bmod a, \,y+1\big) \quad\text{when }y=0,\\
      \mathbf{v}(x,y)&\le \mathbf{v}\big((x+\ell)\bmod a, \,y-1\big) \quad\text{when }y=a-1.
    \end{align*}
Further, we say that two local minima, $\mathbf{v}(x,y)$ and $\mathbf{v}(x',y')$, are \emph{consecutive} if there is no local minimum $\mathbf{v}(\widetilde{x},\widetilde{y})$ with $y<\widetilde{y}<y'$.
\end{definition}

Certainly in light of Lemma~\ref{lem3}, we just need to study the local minima among the $\mathbf{v}$-values, as was done in~\cite[Lem.~5]{Tripathi}. However the statement in \cite[Lem.~5]{Tripathi} is not quite correct, so we provide the corrected version below together with some extra information:

\begin{lemma} \label{lem:5} 
 Suppose $\ell>k$, and let $\mathbf{v}(x,y)$ and $\mathbf{v}(x',y')$ be consecutive local minima. Then $x,x'<\min\{a-\ell,\ell\}$. 
   
   Furthermore, letting  $\rho_{x}\in\{0,1,\ldots,a-\ell-1\}$ be such that $\rho_x\equiv  x-\ell \bmod {(a-\ell)}$, we have
   \[
   \mathbf{v}(x',y')-\mathbf{v}(x,y)=\mathbf{v}\left( (\rho_x \bmod \ell)-x, \left\lceil \frac{\ell-x}{a-\ell}\right\rceil +   \left\lfloor \frac{\ell+\rho_x}{\ell}\right\rfloor\right).
   \] 
\end{lemma}

\begin{proof}
We follow the set-up in \cite[Lem.~5.11]{Peter}. Starting from $\mathbf{v}(x,y)$, note that as we progress through consecutive $\mathbf{v}$-values, we increase the $x$-component by $a-\ell$, working modulo $a$, and we increase the $y$-component by 1. For convenience, we write 
\[
x_i=x+i(a-\ell) \pmod a
\]
for the $x$-component of the $i$th $\mathbf{v}$-value after $\mathbf{v}(x,y)$.
Observe that these consecutive $\mathbf{v}$-values are always increasing until the first time the $x$-component is greater or equal to $\ell$, after which the $\mathbf{v}$-value decreases. Indeed,  recalling that $\ell>k$, and in particular $b\ell >c$ by Lemma~\ref{lem:c-bl}, if we suppose that the first time  the $x$-component is greater or equal to $\ell$ happens at  $x_\iota$,  then 
\[
\mathbf{v}(x_{\iota+1}, y+\iota+1)-\mathbf{v}(x_{\iota}, y+\iota)=-b\ell +c< 0.
\]
We note also that as $\ell\le x_\iota\le a-1$, it follows that $0\le x_{\iota+1}\le a-\ell-1$. The first statement thus follows.

We compute that the number of consecutive increasing $\mathbf{v}$-values, starting from $\mathbf{v}(x,y)$,  is
$$\iota=\left\lceil\frac{\ell-x}{a-\ell}\right\rceil.
$$

If $x_{\iota+1}\ge \ell$, then $\mathbf{v}(x_{\iota+2}, y+\iota+2)$ will likewise be strictly smaller than $\mathbf{v}(x_{\iota+1}, y+\iota+1)$. We deduce that the number of consecutive decreases from $\mathbf{v}(x_{\iota}, y+\iota)$  is equal to 
$$\left\lfloor \frac{x_\iota}{\ell}\right\rfloor.$$
From the definition of $\rho_x$ we have
\begin{equation*}
    \label{eq:r_{x_0}}
\ell-x=\left\lceil\frac{\ell-x}{a-\ell}\right\rceil(a-\ell)-\rho_x
\end{equation*}
and so, as $x_{\iota}=x+\left\lceil\frac{\ell-x}{a-\ell}\right\rceil(a-\ell)$, we obtain
\[
x_{\iota}=\ell-x+\rho_x+x=\ell+\rho_x=\left\lfloor \frac{\ell+\rho_x}{\ell}\right\rfloor\ell+ (\rho_x \bmod \ell).
\]
Thus the number of consecutive decreases from $\mathbf{v}(x_{\iota}, y+\iota)$ is 
$$
\left\lfloor\frac{\ell+\rho_x}{\ell}\right\rfloor.
$$
Therefore, to prove the final statement, if the next local minimum is at $\mathbf{v}(x',y')$, we first note that 
\[
y'-y=\left\lceil \frac{\ell-x}{a-\ell}\right\rceil+\left\lfloor\frac{\ell+\rho_x}{\ell}\right\rfloor.
\]
Similarly, from the above we compute that 
\begin{align*}
x'-x&=\left\lceil \frac{\ell-x}{a-\ell}\right\rceil(a-\ell)-\left\lfloor\frac{\ell+\rho_x}{\ell}\right\rfloor\ell\\
&=\ell-x+\rho_x+(\rho_x \bmod \ell)-\ell-\rho_x\\
&=(\rho_x\bmod \ell) -x,
\end{align*}
as required.
\end{proof}

We remark that the statement in \cite[Lem.~5.11]{Peter} is different from the lemma above, as \cite[Lem.~5.11]{Peter} still had some small inconsistencies that needed to be dealt with.

\begin{remark}
    From \cite[Lem.~4]{Tripathi} and since $b\ell\equiv c\pmod a$, it follows that the difference between consecutive $\mathbf{v}$-values is a non-zero multiple of~$a$. Hence, starting from, and including, $\mathbf{v}(x,0)$, all consecutive  $\mathbf{v}$-values lie in $[bx]$, the congruence class of $bx$ modulo~$a$.
\end{remark}

We recall here Lemma~7 of \cite{Tripathi}. Although the proof of this lemma quotes the incorrect Lemma~5 of \cite{Tripathi}, the statement still holds. For completeness, we give an amended proof, where we first recall some notation from \cite{Tripathi}.

\begin{notation}
    \label{not:m}
For $x\in\mathbb{N}$ we denote by $$\mathbf{m}(x)=\min \{z\in \mathbb{N}
\mid z=bn_1+cn_2 \text{ for some }n_1,n_2\in\mathbb{N}\cup\{0\}\quad\text{and}\quad z\equiv x \pmod a\}.
$$
Also
\[
q:=\left\lfloor\frac{a}{a-\ell}\right\rfloor,\quad r:=a-q(a-\ell).
\] 
\end{notation}

\begin{remark}
Note also that $r\equiv a \bmod {(a-\ell)}$ and $r\equiv \ell \bmod {(a-\ell)}$; indeed, it follows from the equality $a=\ell+a-\ell$. Furthermore $r\le \ell$.
\end{remark}

We could also assume that $r>0$, since the case $r=0$ is equivalent to having $\ell=a-1$, and this case was settled by Brauer and Shockley \cite[Pg.~220]{Brauer}, as was remarked in \cite[Rmk.~2]{Tripathi}. However the proofs in this paper also work for the case $r=0$, so there is no need to exclude this case.

\begin{lemma} \cite[Lem.~7]{Tripathi}
\label{lem:7} Suppose $\ell>k$. For $x\in\{1,\ldots, a-1\}$,  $$\mathbf{m}(bx)=\min\{bx,\mathbf{m}(bx')+cy'\},$$ where
    \[
    (x',y')=\begin{cases}
        \big(x \bmod {(a-\ell)} -r+a-\ell,\, q-\lfloor \tfrac{x}{a-\ell}\rfloor+1\big) & \text{if }0\le x\bmod {(a-\ell)} \le r-1;\\
         \big(x \bmod {(a-\ell)}-r,\, q-\lfloor \tfrac{x}{a-\ell}\rfloor\big) & \text{if }r\le x\bmod {(a-\ell)} \le a-\ell.
    \end{cases}
    \]
\end{lemma}

\begin{proof}
    As observed in the proof of the previous result, the first decrease in the $\mathbf{v}$-values, starting from $\mathbf{v}(x,0)$, is at $\mathbf{v}(x',y')$, for $(x',y')$ given as in the above statement. Hence, the minimum of the consecutive $\mathbf{v}$-values starting from $\mathbf{v}(x,0)$, is either $\mathbf{v}(x,0)$ or the minimum among all consecutive $\mathbf{v}$-values starting from $\mathbf{v}(x',y')$. As all consecutive $\mathbf{v}$-values starting from $\mathbf{v}(x',y')$ can also be respectively expressed as the sum of $cy'$ with all consecutive $\mathbf{v}$-values starting from $\mathbf{v}(x',0)$,  the result now follows, noting that the minimum among all consecutive $\mathbf{v}$-values starting from $\mathbf{v}(x',0)$ is denoted by $\mathbf{m}(bx')$.
\end{proof}

Of course, with Lemma~\ref{lem:5}, i.e. the corrected version of \cite[Lem.~5]{Tripathi}, we could also strengthen the previous lemma by replacing $(x',y')$ with the coordinates of the next local minimum after $\mathbf{v}(x,0)$. However, this is not necessary for the purpose of proving the main result.

\smallskip
 
To proceed, we now split into two cases, when firstly $br<cq$ (Section~\ref{sec:less-than}) and secondly when $br>cq$ (Section~\ref{sec:greater-than}). Note that $br\ne cq$, as remarked in \cite[Subsec.~3.5]{Tripathi}. Indeed, we have $\text{gcd}(b,c)=1$ and $r<a-\ell<a<c$.

\section{The case $br< cq$}\label{sec:less-than}

We first document a useful result.

\begin{lemma}
    \label{lem:minima-increase}
    Let $\mathbf{v}(x,y)$ and $\mathbf{v}(x',y')$ be consecutive local minima, with     $y<y'$. If $\ell>k$ and $br<cq$, then $a-\ell\le \ell$ and $\mathbf{v}(x',y')>\mathbf{v}(x,y)$.
\end{lemma}

\begin{proof}
   First note that the assumption $br<cq$ forces $r<\ell$. Indeed, as $r\le \ell$, if $r=\ell$ then $q=1$ and so $br=b\ell<c$, contradicting Lemma~\ref{lem:c-bl}. Then, as $r\equiv \ell \bmod {(a-\ell)}$, it follows that $r\le \ell -(a-\ell)$. In particular, this implies that $a-\ell\le \ell$.  Indeed, if $\ell<a-\ell$, as $r\equiv \ell\bmod {(a-\ell)}$ and also $r<a-\ell$, we would then have $r=\ell$, a contradiction to the above.
   
   In the notation of the proof of Lemma~\ref{lem:5}, observe that the number of increasing $\mathbf{v}$-values, starting from $\mathbf{v}(x,y)$, is either $q$ or $q-1$, i.e.
   \[
\iota=\left\lceil\frac{\ell-x}{a-\ell}\right\rceil=\begin{cases}
    q & \text{if }x<r,\\
      q-1 & \text{if }x\ge r.
\end{cases}
   \]
   Thereafter, the number of decreasing $\mathbf{v}$-values before the local minimum $\mathbf{v}(x',y')$ is reached, is
$$
\left\lfloor\frac{\ell+\rho_x}{\ell}\right\rfloor,
$$
where $\rho_x\equiv x-\ell \bmod {(a-\ell)}$. Recall that
\[
\ell-x=\left\lceil \frac{\ell-x}{a-\ell}\right\rceil(a-\ell)-\rho_x\]
and so
\[
\ell+\rho_x=\begin{cases}
    q(a-\ell)+x & \text{if }x<r,\\
      (q-1)(a-\ell)+x & \text{if }x\ge r.
\end{cases}
\]
Observe also that
\[
(q-1)(a-\ell)\le \ell< q(a-\ell)\le 2\ell;
\]
indeed, if $q(a-\ell)\le \ell$, then $(q+1)(a-\ell)\le a$, contradicting the definition of $q$. Therefore, for the case $x \geq r$, clearly we have
\[
\left\lfloor \frac{(q-1)(a-\ell)+x}{\ell}\right\rfloor=1
\]
since $x<a-\ell$ by Lemma~\ref{lem:5}. For the case $x<r$, 
we claim that $q(a-\ell)+r\le 2\ell$. Indeed, if $a=q(a-\ell)+r> 2\ell$, then $a-\ell> \ell$, contrary to what we established above. Hence, 
\begin{align*}
\mathbf{v}(x',y')-\mathbf{v}(x,y)&=\begin{cases}
    bq(a-\ell)-b\ell +c(q+1) & \text{if }x<r,\\
      b(q-1)(a-\ell)-b\ell +cq & \text{if }x\ge r,
      \end{cases}\\
      &=\begin{cases}
    b(a-r)-b\ell +c(q+1) & \text{if }x<r,\\
      b(a-r-a+\ell)-b\ell +cq & \text{if }x\ge r,
\end{cases}\\
 &=\begin{cases}
    b(a-\ell) +c-br+cq & \text{if }x<r,\\
      -br+cq & \text{if }x\ge r,
\end{cases}\\
&>0,
\end{align*}
as required.
\end{proof}

\begin{remark}\label{rmk:failure-vs-truth}
    As seen in the proof above, since there is only one decreasing step among the consecutive $\mathbf{v}$-values considered, \cite[Lem.~5]{Tripathi} holds in this case $br<cq$. Therefore \cite[Lem.~5]{Tripathi} is only inconsistent in the case $br>cq$. Indeed, the original formulation in \cite[Lem.~5]{Tripathi} is as follows: \textit{for $\mathbf{v}(x,y)$ and $\mathbf{v}(x',y')$ consecutive local minima, with $0\le x,x'<a-\ell$  (and such that  $y<y'$),}
    \[
    \mathbf{v}(x',y')-\mathbf{v}(x,y)=\begin{cases}
        \mathbf{v}(a-\ell-r,q+1) &\textup{if }0\le x<r,\\
         \mathbf{v}(-r,q) &\textup{if }r\le x<a-\ell.
    \end{cases}
    \]
    As was given in \cite[Example~5.13]{Peter}, upon taking $a=11$, $b=15$ and $c=16$, we have $q=1$ and $\ell=r=4$. Hence the statement in \cite[Lem.~5]{Tripathi} says that the next local minimum after $\mathbf{v}(1,0)=15$ should be $\mathbf{v}(4,2)=92$, however the $\mathbf{v}$-value after $\mathbf{v}(4,2)$, that is $\mathbf{v}(0,3)=48$, is the actual local minimum. As remarked earlier, we confirm here that $br=60 > cq=16$.
\end{remark}

As with \cite[Lem.~7]{Tripathi}, the statement of \cite[Thm.~3]{Tripathi} still holds. The proof is the same as in \cite{Tripathi}, only we replace the reference to  \cite[Lem.~5]{Tripathi} with Lemma~\ref{lem:minima-increase}. For convenience, we include the full proof here, and we have also added some extra justification for clarity. 

\begin{theorem}\cite[Thm.~3]{Tripathi} 
\label{thm:3} If $\ell >k$ and $br<cq$, then writing $\lambda:=\lfloor\tfrac{cq-br}{b(a-\ell)+c}\rfloor$,
\[
g(a,b,c)+a=\begin{cases}
    b\big((\lambda+1)(a-\ell)+r-1\big) &\text{if }\lambda\ge\frac{c(q-1)-br}{b(a-\ell)+c};\\
    b(a-\ell-1)+c(q-\lambda-1) &\text{if }\lambda\le\frac{c(q-1)-br}{b(a-\ell)+c}.
\end{cases}
\]
\end{theorem}

\begin{proof}     
In light of Lemma~\ref{lem:minima-increase}, it follows that  $\mathbf{m}(bx)=bx$ for $x\in\{1,\ldots,a-1\}$. Then from Lemma~\ref{lem:7} and using its notation, we have $\mathbf{m}(bx)=\min\{bx,bx'+cy'\}$ for $x\in\{1,\ldots,a-1\}$.

    We fix some $x\in\{1,\ldots,a-1\}$, and write $m:=\lfloor\tfrac{x}{a-\ell}\rfloor$ and $s:=x\bmod (a-\ell)$. We set $\epsilon$ to be 0 or 1 according to whether $s\ge r$ or $s<r$ respectively. From Lemma~\ref{lem:7}  we deduce that $\mathbf{m}(bx)=bx$ if and only if
    \begin{equation}\label{eq:lambda}
    \begin{split}
     &  \qquad\qquad b\big(m(a-\ell)+s\big)=bx\le bx'+cy'=b\big(s-r+\epsilon(a-\ell)\big)+c\big(q-(m-\epsilon)\big)\\
     &\iff \quad b\big((m-\epsilon)(a-\ell)+r\big) \le c\big(q-(m-\epsilon)\big)\\
      &\iff \quad (m-\epsilon)\big(b(a-\ell)+c\big) \le cq-br\\
          &\iff \qquad\qquad \qquad\qquad \,\,\quad m\le \lambda+\epsilon.
    \end{split}
    \end{equation}
    Thus for fixed $s\in\{0,1,\ldots,a-\ell-1\}$, we assert that
    \[
    \max_{m} \mathbf{m}(bx)=\max\left\{b\big((\lambda+\epsilon)(a-\ell)+s\big),\, b\big(\epsilon(a-\ell)+s-r\big)+c(q-\lambda-1)\right\}.
    \]
    Indeed, for $m\in\{0,1,\ldots,\lambda+\epsilon\}$, the above equivalences yield that $\mathbf{m}(bx)=bx$, and the maximum of these is when $m=\lambda+\epsilon$. For $m\ge \lambda+\epsilon+1$, we have $\mathbf{m}(bx)=bx'+cy'$. As $x'$ does not depend on $m$ and as $y'=q-(m-\epsilon)$, we see that as $m$ increases, the corresponding value $bx'+cy'$ decreases. Hence the maximum of these values occurs when $m-\epsilon=\lambda+1$.

    Now note that the difference between the above two terms,
    \begin{equation}\label{eq:difference}
    \begin{split}
       &\left[b\big(\epsilon(a-\ell)+s-r\big)+c(q-\lambda-1)\right]-b\big((\lambda+\epsilon)(a-\ell)+s\big)\\
    &\qquad\qquad\qquad\qquad\qquad\qquad\qquad\qquad\qquad\qquad=c(q-1)-br-\lambda\big(b(a-\ell)+c\big),
    \end{split}
    \end{equation}
     is independent of $s$ and $\epsilon$. Note also that
    $$b\big(\lambda(a-\ell)+a-\ell-1\big) \le b\big((\lambda+1)(a-\ell)+r-1\big).
    $$
    Hence, we claim that if the difference in \eqref{eq:difference} is non-positive, i.e. if $\lambda\ge\frac{c(q-1)-br}{b(a-\ell)+c}$, then
    \[
    \max_{s} \left\{\max_m\mathbf{m}(bx)\right\}=b\big((\lambda+1)(a-\ell)+r-1\big),
    \]
    and  if the difference in \eqref{eq:difference} is non-negative, i.e. if $\lambda\le\frac{c(q-1)-br}{b(a-\ell)+c}$, then
    \[
    \max_{s} \left\{\max_m\mathbf{m}(bx)\right\}=b(a-\ell-1)+c(q-\lambda-1);
    \]
    indeed, the maximum among $b\big(a-\ell+s-r\big)$ for $s<r$ occurs when $s=r-1$, and the maximum among $b\big(s-r\big)$ for $s\ge r \ge 0$ occurs when $s-r=a-\ell -1$.

   Lastly, observe that $b\big((\lambda+1)(a-\ell)+r-1\big)=b(a-\ell-1)+c(q-\lambda-1)$ if and only if $\lambda\big(b(a-\ell)+c\big)=c(q-1)-br$, so the constraints on $\lambda$ in the statement of the theorem are justified.
 \end{proof}

As mentioned in the introduction,  Theorem 4  of \cite{Tripathi} can be proved with similar adjustments.

\section{The case $br> cq$}\label{sec:greater-than}

We next turn to Lemma~9 of~\cite{Tripathi}. The proof is again the same as in \cite{Tripathi}, only we need to slightly rephrase any references to \cite[Lem.~5]{Tripathi}. For convenience, as before we include the full proof here, and we also provide more justification for certain claims. Recall also, that when considering consecutive $\mathbf{v}$-values, this is equivalent to the change from a coordinate $(x,y)$ to $(x+a-\ell,y+1)$. This change in coordinates is referred to as  a \emph{step} of size~1.
 
For brevity, we shall also write $$u:=(a-\ell) \bmod r.$$ Note that in particular $r>0$ as $br>cq$.
\begin{lemma}\cite[Lem.~9]{Tripathi} \label{lem:9}
If $\ell>k$ and $br>cq$, then
\[
\max_{0\le x<a} \mathbf{m}(bx)=\max \left\{\max_{0\le x<u}\mathbf{m}(bx)+cq,\, \max_{u\le x<r}\mathbf{m}(bx)\right\}+cq\left\lfloor\frac{a-\ell-1}{r}\right\rfloor.
\]
\end{lemma}

\begin{proof}
For $1\le x<a$, for conciseness we again write $m:=\lfloor \tfrac{x}{a-\ell}\rfloor$ and $s:=x\bmod {(a-\ell)}$. Note that we cover every element in the set $\{1,\ldots,a-1\}$ when we let $m$ and $s$ vary in the expression $m(a-\ell)+s$. 

Since $br>cq$, referring to the notation in Theorem~\ref{thm:3} we have $\lambda<0$. Thus, if $m\ge 1$ then  certainly $m>\lambda+1$. Hence, referring to~\eqref{eq:lambda} and Lemma~\ref{lem:7}, 
we have $\mathbf{m}(bx)=\mathbf{m}(bx')+cy'$; here $x',y'$ are as in Lemma~\ref{lem:7}.   

So we first consider
\begin{equation}\label{eq:m>1}
\begin{split}
& \max_{m\ge 1}\left\{\max_{r\le s< a-\ell }\mathbf{m}\big(b(s-r)\big)+c(q-m),\,\max_{0\le s< r }\mathbf{m}\big(b(s+a-\ell-r)\big)+c(q-m+1)\right\} \\
    &\qquad =  \max_{m\ge 1}\left\{\max_{0\le s< a-\ell-r }\mathbf{m}(bs)+c(q-m),\,\max_{a-\ell-r\le s< a-\ell }\mathbf{m}(bs)+c(q-m+1)\right\} \\
    &\qquad=  \max\left\{\max_{0\le x< a-\ell-r }\mathbf{m}(bx)+c(q-1),\,\max_{a-\ell-r\le x< a-\ell }\mathbf{m}(bx)+cq\right\}. 
\end{split}
\end{equation}
We next compare the above with the maximum value among $\mathbf{m}(bx)$ for when $m=0$. This surmounts to comparing the above expression with 
\[
\max_{0\le x<a-\ell}\mathbf{m}(bx).\]
As, from the final line of \eqref{eq:m>1}, we have
\begin{align*}
&\max\left\{\max_{0\le x<a-\ell}\mathbf{m}(bx),\,\max\Big\{\max_{0\le x< a-\ell-r }\mathbf{m}(bx)+c(q-1),\,\max_{a-\ell-r\le x< a-\ell }\mathbf{m}(bx)+cq\Big\}\right\}\\
&\qquad=\max\left\{\max_{0\le x< a-\ell-r }\mathbf{m}(bx)+c(q-1),\,\max_{a-\ell-r\le x< a-\ell }\mathbf{m}(bx)+cq\right\},
\end{align*}
we hence obtain, also  from using Lemma~\ref{lem:7},  that
\[
  \max_{0\le x<a} \mathbf{m}(bx)=  \max\left\{\max_{0\le x< a-\ell-r }\mathbf{m}(bx)+c(q-1),\,\max_{a-\ell-r\le x< a-\ell }\mathbf{m}(bx)+cq\right\}. 
\]

\smallskip

So from the above line, it suffices to consider $0\le x<a-\ell$. As stated in the proof of Lemma~\ref{lem:minima-increase}, if $x\ge r$, starting from $\mathbf{v}(x,y)$, the first decrease in the $\mathbf{v}$-values occurs after $q$ steps, and results in the change of coordinates from $(x,y)$ to $(x-r,y+q)$. As $br>cq$, we observe that $\mathbf{v}(x-r,y+q)<\mathbf{v}(x,y)$. Hence, for each $ x'\in\{0,1,\ldots,r-1\}$, there is  a  number $x\in\{r,r+1,\ldots,2r-1\}$, namely $x'+r$, with 
 $bx=\mathbf{v}(x,0)>\mathbf{v}(x',q)=bx'+cq$. Note that $x'$ here corresponds to the $x'$ in Lemma~\ref{lem:7} and there $y'=q$. Hence we deduce from  Lemma~\ref{lem:7} that $\mathbf{m}(bx)=\mathbf{m}(bx')+cq>\mathbf{m}(bx')$.  Therefore
 \[
 \max_{0\le x<a}\mathbf{m}(bx)=\max\left\{\max_{r\le x< a-\ell-r }\mathbf{m}(bx)+c(q-1),\,\max_{a-\ell-r\le x< a-\ell }\mathbf{m}(bx)+cq\right\}. 
 \] Repeating this consideration for each $ x'\in\{r,r+1,\ldots,2r-1\}$, etc, we see that for the purpose of seeking the desired maximum as displayed above, we can restrict to the range $a-\ell-r\le x< a-\ell$.

Now if $nr$ is the unique multiple of $r$ satisfying $a-\ell-r\le nr<a-\ell$, then $n-1$ such steps, as above, of size $q$ can be applied for those $x<nr$ and $n$ such steps of size $q$ for $x\ge nr$ in the range  $a-\ell-r\le x< a-\ell$. Therefore, from the above,
\begin{align*}
    \max_{0\le x<a}\mathbf{m}(bx)&=\max_{a-\ell-r\le x<a-\ell}\mathbf{m}(bx)+cq\\
    &=\max\left\{\max_{0\le x<u}\mathbf{m}(bx)+cq ,\, \max_{u\le x <r}\mathbf{m}(bx)\right\}+cq\left\lfloor\frac{a-\ell-1}{r}\right\rfloor,
\end{align*}
as required. 
\end{proof}

We next recall the following definition from \cite{Tripathi}, which will play an important role in the sequel.
\begin{definition}
    \cite[Def.~3]{Tripathi}\label{def:3}
    Let $\ell>k$ and $br>cq$. We define 
    \[
    \mathscr{X}:=\{x\in\{1,\ldots,r\}\mid \mathbf{m}(bx)\equiv 0\bmod c\}.
    \]
\end{definition}

We recall the following observations from  \cite{Tripathi}, where we also fill in a minor gap in the proof of the first statement  in the lemma below; cf. Remark~\ref{rmk:only-two-local-minima}:
\begin{lemma} \cite[Rmk.~6]{Tripathi}\label{lem:rmk6}
    Suppose $\ell>k$ and $br>cq$. Then $r\in \mathscr{X}$ and 
    \[
    \min \mathscr{X}=\min\{x\in\{1,\ldots,r\}\mid \mathbf{m}(bx)\ne bx\}.
    \]
\end{lemma}

\begin{proof}
  We claim that $\mathbf{m}(br)=cq$, which then yields the first statement. As $a=q(a-\ell)+r$, we have the following inequalities among the first $q+1$ consecutive $\mathbf{v}$-values starting from $ \mathbf{v}(r,0)$:
   \[
   \mathbf{v}(r,0)< \mathbf{v}(a-\ell+r,1)<\cdots < \mathbf{v}((q-1)(a-\ell)+r,q-1) > \mathbf{v}(0,q)
   \]
   So $ \mathbf{v}(r,0)$ and $ \mathbf{v}(0,q)$ are consecutive local minima. As $br>cq$, we obtain that 
   \[
    \min \{\mathbf{v}(r,0),\mathbf{v}(0,q)\} = \mathbf{v}(0,q).
   \]
   We now note that any local minimum $\mathbf{v}(x,y)$ with $y'>q$, so occurring after $\mathbf{v}(0,q)$, will always satisfy $\mathbf{v}(x,y)>\mathbf{v}(0,q)$. Hence the claim.

For the last statement, write $\widehat{x}:=\min \mathscr{X}$ and $\overline{x}=\min\{x\in\{1,\ldots,r\}\mid \mathbf{m}(bx)\ne bx\}$. Observe that $\mathbf{m}(b\overline{x})=bx_0+cy_0$ for some $x_0\ge 0$ and $y_0\ge 1$. Hence
\[
bx_0<bx_0+cy_0<b\overline{x},
\]
and therefore $x_0<\overline{x}$. Now we also have that $cy_0\in [b(\overline{x}-x_0)]$ and that $cy_0<b(\overline{x}-x_0)$, which gives that $\mathbf{m}(b(\overline{x}-x_0))<b(\overline{x}-x_0)$. If  $x_0>0$, this then contradicts the choice of $\overline{x}$ as the minimal element of the set  $\{x\in\{1,\ldots,r\}\mid \mathbf{m}(bx)\ne bx\}$. So we must have $x_0=0$.  Hence $\mathbf{m}(b\overline{x})=cy_0$, and so $\overline{x}\in\mathscr{X}$. As $\mathscr{X}\subseteq \{x\in\{1,\ldots,r\}\mid \mathbf{m}(bx)\ne bx\}$, it follows that $\overline{x}=\widehat{x}$, as wanted.
\end{proof}

\begin{remark}\label{rmk:only-two-local-minima}
    It was claimed in \cite[Rmk.~6]{Tripathi}  that  the local minima in the class $[br]$ consist of $\mathbf{v}(r,0)$ and $\mathbf{v}(0,q)$ only. As indicated in the first part of the above proof, there can be local minima in the class  $[br]$ after $\v(0,q)$, and it is not difficult to find examples of such. It was most likely  meant in \cite{Tripathi} that there are no other local minima between $br$ and $cq$, since we go from $\v(r,0)$ to $\v(0,q)$ in one $\downarrow$-step.
\end{remark}

To proceed, we recall some notation from \cite[Def.~4]{Tripathi}; however we amend one of the definitions (cf. Remark~\ref{rmk:new-mu} below):
\begin{notation}\label{notation}
    Set $A:=br-cq$, $B:=b(a-\ell-r)+c(q+1)$, and
    \[
    \Lambda:=\left\lfloor\frac{r}{a-\ell-r}\right\rfloor,\qquad    \Delta:=\left\lfloor\frac{A}{B}\right\rfloor,\qquad\Lambda':=\left\lfloor\frac{a-\ell-r}{r}\right\rfloor,\qquad    \Delta':=\left\lfloor\frac{B}{A}\right\rfloor.
    \]
    Also    let
  $$ \mu := \mathbf{\min}\left\{i \in \mathbb{Z}_{\ge 0}\, \bigg\vert \left\lfloor \frac{(i+1)B}{A}\right\rfloor \neq \left\lfloor \frac{(i+1)(a-\ell-r)}{r}\right\rfloor\right\},$$
and recall that $u\equiv (a-\ell) \bmod r$.
\end{notation}

\begin{remark}\label{rmk:new-mu}
    In \cite[Lem.~11]{Tripathi}, the quantity  $\mu$, written $\mu'$ in \cite{Tripathi}, was instead defined as
$$
\mu:=\max\left\{i\in\mathbb{Z}_{\ge 0}\,\bigg\vert\, \left\lfloor \frac{i B}{A}\right\rfloor= \left\lfloor \frac{i(a-\ell-r)}{r}\right\rfloor\right\}.
$$
We suspect that our definition of $\mu$ is actually what Tripathi had in mind,  since $\mu$ was stated to be $4$ in \cite[Example~5]{Tripathi}, which follows our definition, whereas Tripathi's definition should give $\mu \ge 9$. This can be seen by plugging in the example's values with $i = 9$, and noting that
$$
\left\lfloor\frac{9 \cdot 2500}{3800}\right\rfloor = 5 = \left\lfloor\frac{9 \cdot 22}{39}\right\rfloor
.$$
For a more detailed example of the difference in the two definitions, c.f. \cite[Example~5.18]{Peter}.

\end{remark}

Before proceeding, we document a useful observation.

\begin{lemma}\label{lem:A=B}
     Suppose $\ell>k$ and $br>cq$. Then, with reference to Notation~\ref{notation}, we have $A\ne B$.
\end{lemma}

\begin{proof}
    Suppose for a contradiction that $A=B$, and so $br-cq=b(a-\ell-r)+c(q+1)$. Rearranging gives
    \[
    b(2r-a+\ell)=c(2q+1).
    \]
    As $\text{gcd}(b,c)=1$, we have in particular that $b$ divides $2q+1$. Hence $a\le 2q$. From $a=q(a-\ell)+r$, it follows that $a-\ell<2$. That is, we have $a-\ell=1$. As remarked just before Lemma~\ref{lem:7}, this is equivalent to the case $r=0$ which cannot happen when $br>cq$.
\end{proof}

The corresponding statement of \cite[Lem.~11]{Tripathi} still holds with minor adjustments to the proof in light of the new definition of  $\mu$, but also because the reference to \cite[Lem.~5]{Tripathi} is only used to determine how many steps are needed to go between two specific local minima of the form $\mathbf{v}(x,0)$ and $\mathbf{v}(0,y)$, regardless of how many other (fake) local minima one needs to pass through. As these local minima  $\mathbf{v}(x,0)$ and $\mathbf{v}(0,y)$ are  always true local minima, there is no issue with this aspect the proof of  \cite[Lem.~11]{Tripathi}. Indeed, the use of \cite[Lem.~5]{Tripathi} here just involves potentially passing through some fake local minima. Specifically, to go from a fake local minimum to the next true local minimum, one has to implement more $\downarrow$-steps, where in the notation of \cite{Tripathi} (see the remarks  after \cite[Lem.~5]{Tripathi}) a \emph{$\downarrow$-step} is the change $(x,y)\mapsto (x-r,y+q)$ in the coordinates of the $\mathbf{v}$-values, which results in a decrease in the $\mathbf{v}$-value by $A$ (cf. Notation~\ref{notation}). Analogously an \emph{$\uparrow$-step} is the change $(x,y)\mapsto (x+a-\ell-r,y+q+1)$,  which results in an increase in the $\mathbf{v}$-value by $B$. In particular, we note the following, which gives more information on when fake local minima occur (compare also with Remark~\ref{rmk:fake-occurrence}):

\begin{lemma}\label{lem:failure}
 Suppose $\ell>k$ and $br>cq$, and let $\mathbf{v}(x,y)$ and $\mathbf{v}(x',y')$ be consecutive local minima. Then either 
   \[
   y'-y=\begin{cases}
       q+1 & \text{if }0\le x<r,\\
       q & \text{if }r\le x<a-\ell,
   \end{cases}
\quad\text{and then}
   \quad
   x'-x=\begin{cases}
       a-\ell-r & \text{if }0\le x<r,\\
       -r & \text{if }r\le x<a-\ell,
   \end{cases}
   \] 
   or
    \[
   y'-y>\begin{cases}
       q+1 & \text{if }0\le x<r,\\
       q & \text{if }r\le x<a-\ell,
   \end{cases}
   \] 
   and  in this case $q=1$ and $r=\ell$.
\end{lemma}

\begin{proof}
    As in the proof of Lemma~\ref{lem:5}, after a number of consecutive $\mathbf{v}$-value increases from $\mathbf{v}(x,y)$, the number of decreasing $\mathbf{v}$-values, before the next local minimum $\mathbf{v}(x',y')$ is reached, is
$$
\left\lfloor\frac{\ell+\rho_x}{\ell}\right\rfloor,
$$
where $\rho_x\equiv x-\ell \bmod {(a-\ell)}$ and
\[
\ell+\rho_x=\begin{cases}
    q(a-\ell)+x & \text{if }x<r,\\
      (q-1)(a-\ell)+x & \text{if }x\ge r.
\end{cases}
\]
As argued in the proof of Lemma~\ref{lem:minima-increase},
\[
(q-1)(a-\ell)\le \ell< q(a-\ell).
\]
If $q(a-\ell)+r\le 2\ell$, then the number of decreasing consecutive $\mathbf{v}$-values to get to the local minimum $\mathbf{v}(x',y')$ is just one, and in this case we obtain the first statement, as argued in \cite[Proof of Lem.~5]{Tripathi}. Suppose otherwise, and so $q(a-\ell)+r> 2\ell$, equivalently $a-\ell>\ell$, and so one passes through at least two decreasing consecutive $\mathbf{v}$-values to arrive at the local minimum $\mathbf{v}(x',y')$. As $(q-1)(a-\ell)\le \ell$, it then follows that $q=1$, and thus $r=\ell$. 
\end{proof}

\begin{remark}\label{rmk:fake-occurrence}
Continuing off Remark~\ref{rmk:failure-vs-truth}, we see from the above that when \cite[Lem.~5]{Tripathi} does not hold, it follows that $q=1$ and $r=\ell$. Therefore, when we need to perform more than one decreasing step to get to the next local minimum, each decreasing step is actually a $\downarrow$-step.
\end{remark}

\begin{remark}
In fact the converse to the above remark is true,  i.e. if $q=1$ and $r=\ell$, then \cite[Lem.~5]{Tripathi} does not hold. Indeed, consider  $x_0=\ell -1$. One increasing step yields $x_1 = a-1$, and one decreasing step gives $x_2 = a-\ell-1$. Recalling that $r=\ell<a-\ell$, we see that there is room for at least one more decreasing step.
\end{remark}

Below we state \cite[Lem.~11]{Tripathi} using the new definition of~$\mu$ and we give the amended proof, where we also provide more justification for several claims given in the original proof. For convenience, we first isolate a small result that will be used in the proof of \cite[Lem.~11]{Tripathi}. 

\begin{lemma}
    \label{lem:r-not-div-by-a-l-r}
If $br>cq$, then  $r$ is not divisible by $a-\ell-r$.
\end{lemma}

\begin{proof}
Assume for a contradiction that $r \equiv 0\bmod{(a-\ell-r)}$. This can only happen if either $r=0$, which by  the hypothesis is excluded,  or $r$ is some  positive integer multiple of $a-\ell-r$. So suppose the latter, and let $r = \eta(a-\ell-r)$, for some $\eta \in \mathbb{N}.$ From the definition of~$r$, we thus have  
\begin{align*}
     a &= r + q(a-\ell)\\
    &= \eta(a-\ell-r) + q(a-\ell) \\
    &= \eta(a-\ell-r) + q(a-\ell-r + r) \\
    &= \eta(a-\ell-r) + q\big(a-\ell-r + \eta(a-\ell-r)\big) \\
    &= (\eta+q+q\eta)(a-\ell-r).
\end{align*}
So  $a$ is also a multiple of $a-\ell-r$.  Similarly,
\begin{align*}
    \ell &= a-r-(a-\ell-r)\\
    &= (\eta+q+q\eta)(a-\ell-r) - \eta(a-\ell-r)-(a-\ell-r)\\
    &= (q+q\eta-1)(a-\ell-r).
\end{align*}
This implies that $\gcd(a,\ell)$ is divisible by $a-\ell-r$, which is impossible under the standing assumption of $a,b,c$ being pairwise coprime. Indeed, recall that $\ell \equiv cb^{-1} \bmod a$ and therefore $b\ell=c+ \kappa a$ for some  integer $\kappa\ne 0$. Then this combined with the fact that $\gcd(a,\ell)$ is divisible by $a-\ell-r$, yields that $c$ is divisible by $a-\ell-r$, and hence $\gcd(a,c)$ is divisible by $a-\ell-r$, a contradiction.
\end{proof}

\begin{lemma}
    \cite[Lem. 11]{Tripathi}\label{lem:11}
    Let $\ell>k$ and $br>cq$. Then, with reference to Notation~\ref{notation},
    \begin{align*}
        \mathscr{X}&=\left\{r\left(\left\lfloor\tfrac{(a-\ell-r)t}{r}\right\rfloor+1\right)-(a-\ell-r)t\mid 0\le t\le \mu\right\}\\
        &=\left\{\big(r-(a-\ell-r)t\big) \bmod r\mid 0\le t\le \mu\right\}.
    \end{align*}
    Furthermore, if $\mu<\lfloor\tfrac{r}{u}\rfloor$ then
    \[
    \mathscr{X}=\{r-ut\mid  0\le t\le \mu\}.
    \]
    In particular, if $\Lambda>\Delta$ or $\Delta'>\Lambda'$, then
    \[
     \mathscr{X}=\{r-(a-\ell-r)t\mid  0\le t\le \Delta\}.
    \]
\end{lemma}

\begin{proof}
In light of Lemma~\ref{lem:failure}, we note that to pass from one local minimum to the next, we use one of three options: either an $\uparrow$-step $(x,y)\mapsto (x+a-\ell-r,y+q+1)$ when $0\le x\le r-1$; or a $\downarrow$-step $(x,y)\mapsto (x-r,y+q)$ when $r\le x<a-\ell$; or lastly, at most one $\uparrow$-step 
 followed by several $\downarrow$-steps (recall that in this last option we have $q=1$ and $r=\ell$). Observe that an $\uparrow$-step causes an increase in the $\mathbf{v}$-value by the quantity~$B$ and a $\downarrow$-step causes a decrease in the $\mathbf{v}$-value by the quantity~$A$.

 From Lemma~\ref{lem:rmk6}, we have that $r\in\mathscr{X}$. Consider now $x\in\mathscr{X}\backslash\{r\}$. We observe that then $\mathbf{m}(bx)=cy$ for some $y\in\mathbb{N}$. Since $\mathbf{v}(x,0)$ and $\mathbf{v}(0,y)$ are both local minima in the congruence class $[bx]$, we can arrive at $\mathbf{v}(0,y)$ starting from $\mathbf{v}(x,0)$ via a sequence of $t_1$ $\downarrow$-steps and $t_2$ $\uparrow$-steps, for some $t_1,t_2\in \mathbb{N}$. Indeed, since $0<x<r$, both $t_1$ and $t_2$ must be positive.  So 
 \[
 x=rt_1-(a-\ell-r)t_2\qquad\text{and}\qquad y=qt_1+(q+1)t_2,
 \]
 and thus the inequality $0<x<r$ yields
 \begin{equation}\label{eq:t's}
 \frac{r}{a-\ell-r}(t_1-1)<t_2<\frac{r}{a-\ell-r}t_1.
 \end{equation}

 Let $x=rt_1-(a-\ell-r)t_2$ with $t_1,t_2\in\mathbb{N}$ satisfying \eqref{eq:t's}. We furthermore claim that $x\in\mathscr{X}\backslash\{r\}$ if and only if $t_2\le \mu$.  
First note that for such an $x$, the inequality
 $$\mathbf{v}(x,0)=bx>cy=\mathbf{v}(0,y),
 $$
 where $y=qt_1+(q+1)t_2$ as stated above, rearranges to $At_1>Bt_2$. 
 
 Now to prove the claim, we first suppose that $1\le t_2\le \mu$. Consider any local minimum $\mathbf{v}(x',y')$ in the class $[bx]$ such that $(x',y')\ne (0,y)$ with $x'\ge 0$ and $0\le y'<y$. Suppose we reach $\mathbf{v}(0,y)$ from $\mathbf{v}(x',y')$ in $s_1$ $\downarrow$-steps and $s_2$ $\uparrow$-steps. Then, we obtain $x'=rs_1-(a-\ell-r)s_2>0$, so 
 \[
 s_1> \frac{(a-\ell-r)s_2}{r}.
 \]
  From the definition of $\mu$  and the fact that $s_2\le t_2\le\mu$, we have 
 \[
 \left\lfloor \frac{(a-\ell-r)s_2}{r}\right\rfloor=\left\lfloor \frac{Bs_2}{A}\right\rfloor
 \]
 and so, as
 \[
 \frac{(a-\ell-r)s_2}{r}<\frac{Bs_2}{A},
 \]
 it follows that $\frac{Bs_2}{A}\notin\mathbb{Z}$ 
 and that
 \[
 s_1\ge  \left\lceil \frac{Bs_2}{A}\right\rceil >\frac{Bs_2}{A}.
 \]
 We now get that
 \[
 \mathbf{v}(x',y')-\mathbf{v}(0,y)=As_1-Bs_2=A\left(s_1-\frac{Bs_2}{A}\right)>0.
 \]
 Hence $\mathbf{m}(bx)=cy$, and so $x\in\mathscr{X}$, whenever $t_2\le \mu$ and $0<x<r$ is of the above form. Indeed, recall  that any local minimum $\mathbf{v}(x',y')$ with $y'>y$, so occurring after $\mathbf{v}(0,y)$, will always satisfy $\mathbf{v}(x',y')>\mathbf{v}(0,y)$.
 
 Next, for the reverse direction, suppose that $t_2>\mu$. Note from \eqref{eq:t's} that
 \[
 t_1-1<\frac{(a-\ell-r)t_2}{r}<t_1,
 \]
 so in particular $\tfrac{(a-\ell-r)t_2}{r}\notin\mathbb{Z}$.  Then, from the definition of $\mu$, we have
 \begin{align*}
     t_1&< 1+ \frac{(a-\ell-r)t_2}{r}\\
     &= 1+ \frac{(a-\ell-r)(\mu+1)}{r}+ \frac{(a-\ell-r)(t_2-\mu-1)}{r}\\
     &\le 1+\left\lfloor \frac{(a-\ell-r)(\mu+1)}{r}\right\rfloor + \frac{(a-\ell-r)(t_2-\mu-1)}{r}\\
      &= \left\lfloor \frac{B(\mu+1)}{A}\right\rfloor + \frac{(a-\ell-r)(t_2-\mu-1)}{r}\\
       &< \frac{B(\mu+1)}{A} + \frac{B(t_2-\mu-1)}{A}
 \end{align*}
 and so we see that
 \[
 t_1<\frac{Bt_2}{A}.
 \]
 Hence, as seen above, the inequality $At_1< Bt_2$ is equivalent to $\mathbf{v}(x,0)<\mathbf{v}(0,y)$. Thus $x\notin\mathscr{X}$, as wanted. The claim is now proved.

 Finally, to finish the proof of the first statement, note that $$x=rt_1-(a-\ell-r)t_2=r-(a-\ell-r)t_2+r(t_1-1).$$ Since $\mathscr{X}\subseteq \{1,\ldots,r\}$, we have that $x=\big(r-(a-\ell-r)t_2\big) \bmod r$ for all $x\in\mathscr{X}$. The first equation in the statement now follows, using the equation $$(a-\ell-r)t_2=r\left\lfloor \frac{(a-\ell-r)t_2}{r}\right\rfloor+\big((a-\ell-r)t_2\big) \bmod r.$$

 Next, if $\mu <\lfloor \tfrac{r}{u}\rfloor$, then  $0\le ut<r$ for $0\le t\le \mu$. As $(a-\ell-r)t\equiv ut \bmod r$, it follows that $\mathscr{X}=\{r-ut\mid  0\le t\le \mu\}$ in this case.

 For the final statement, assuming $\Lambda>\Delta$ or $\Delta'>\Lambda'$, it suffices to prove the following two claims:
 \begin{enumerate}
     \item [(i)]   $\mu=\Delta$; and
       \item [(ii)]    $\mu< \lfloor \tfrac{r}{u}\rfloor$.
 \end{enumerate}
 For (i), we first  observe that $\lfloor\tfrac{\Delta B}{A}\rfloor=0$. Indeed,  otherwise we would have $\lfloor\tfrac{\Delta B}{A}\rfloor=1$ and in particular $\eta:=\tfrac{A}{B}\ge 2$ is a positive integer; compare also with Lemma~\ref{lem:A=B}. Consequently $\Delta'=0$ and we cannot have $\Delta'>\Lambda'$. Therefore, by assumption $\Lambda>\Delta=\eta$, which gives $\Lambda=\lfloor \tfrac{r}{a-\ell-r}\rfloor\ge \eta+1$. So $r\ge (\eta+1)(a-\ell-r)$. Returning to $\eta:=\tfrac{A}{B}$, we get
 \[
 br-cq=\eta[b(a-\ell-r)+c(q+1)],
 \]
 equivalently,
 \[
 b[\eta(a-\ell-r)-r]+c[\eta(q+1)+q]=0.
 \]
 Therefore $c$ divides $r-\eta(a-\ell-r)$. On the other hand, since $$
 r\ge (\eta+1)(a-\ell-r)>\eta(a-\ell-r)>0,
 $$
 we have
 \[
0<r-\eta(a-\ell-r)<r<a-\ell<c.
 \]
 This gives the desired contradiction to  $c$ dividing $r-\eta(a-\ell-r)$.  Hence $\lfloor\tfrac{\Delta B}{A}\rfloor=0$ and
 \[
 \lfloor\tfrac{i B}{A}\rfloor=0=\lfloor\tfrac{i (a-\ell-r)}{r}\rfloor
 \]
 for $0\le i \le \Delta$, since $\tfrac{B}{A}>\tfrac{a-\ell-r}{r}>0$. So to prove that $\mu=\Delta$, we need to show that
 \[
 \left\lfloor\tfrac{(\Delta+1)B}{A}\right\rfloor\ne \left\lfloor\tfrac{(\Delta+1)(a-\ell-r)}{r}\right\rfloor.
 \]
We begin by noting that $\lfloor\tfrac{(\Delta +1)B}{A}\rfloor\ge 1$. Then suppose first that  $\Delta'>\Lambda'$. Here $\mu=0$ by definition, and  as $\Delta'>\Lambda'$ implies that $\Delta'>0$, with reference to Lemma~\ref{lem:A=B} we obtain $\Delta=0=\mu$ in this case.

Next, if $\Lambda>\Delta$, then $$\left\lfloor\tfrac{(\Delta+1)(a-\ell-r)}{r}\right\rfloor\le \left\lfloor\tfrac{\Lambda(a-\ell-r)}{r}\right\rfloor.$$
 We likewise claim that $\lfloor\tfrac{\Lambda(a-\ell-r)}{r}\rfloor=0$. Indeed, since $\Lambda = \lfloor\tfrac{r}{a-\ell-r}\rfloor$, we have that
 \begin{align*}
     \left\lfloor\frac{\Lambda(a-\ell-r)}{r}\right\rfloor &= \left\lfloor\frac{\lfloor\tfrac{r}{a-\ell-r}\rfloor(a-\ell-r)}{r}\right\rfloor\\
     &= \left\lfloor\frac{r - \big(r \bmod (a-\ell-r)\big)}{r}\right\rfloor\\ &= 1 - \left\lceil\frac{r \bmod (a-\ell-r)}{r}\right\rceil,
 \end{align*}
 but the last expression is less than 1, since $r \not\equiv 0\bmod{(a-\ell-r)}$ by Lemma~\ref{lem:r-not-div-by-a-l-r}. Therefore $\lfloor\tfrac{\Lambda(a-\ell-r)}{r}\rfloor=0$, as wanted, and  so $\mu=\Delta$ in this second case too. 
 
 Lastly, we show (ii),  which completes the proof. For the case $\Lambda>\Delta$, note that $\Lambda'=0$ as $\Lambda>0$ and $r\not\equiv 0 \bmod {(a-\ell-r)}$. Therefore $u=a-\ell-r-\Lambda'r=a-\ell-r$ and $\mu=\Delta<\Lambda=\lfloor\tfrac{r}{a-\ell-r}\rfloor=\lfloor\tfrac{r}{u}\rfloor$. For the case $\Delta'>\Lambda'$, we recall that $\mu=\Delta=0$. Also we have 
 \begin{equation}\label{eq:u-less-than-r}
 u=a-\ell-r-\Lambda'r< r
 \end{equation}
 since $\tfrac{a-\ell-r}{r}<\Lambda'+1$. We obtain  $\lfloor\tfrac{r}{u}\rfloor>0$, as if $\lfloor\tfrac{r}{u}\rfloor=0$, then \[
 \left\lfloor\frac{r}{a-\ell-(1+\Lambda')r}\right\rfloor=0
 \]
 which gives $(1+\Lambda')r< a-\ell-r$, which is not possible, in light of \eqref{eq:u-less-than-r}. Thus $\mu=\Delta=0<\lfloor\tfrac{r}{u}\rfloor$, as wanted.
\end{proof}

In the case $\Delta'>\Lambda'$ above,  we have from the proof that $\mu=\Delta=0$, and so here $\mathscr{X}=\{r\}$.

\smallskip

We next recall some useful remarks from \cite{Tripathi}.
\begin{lemma}
    \cite[Rmk.~8]{Tripathi} \label{lem:rmk8}
     Suppose $\ell>k$ and $br>cq$. With reference to Notation~\ref{notation}, assume $\mu>\lfloor\tfrac{r}{u}\rfloor$. Then  $\Lambda'=\Delta'=0$, $u=a-\ell-r$, and $\lfloor \tfrac{r}{u}\rfloor=\Lambda=\Delta$. 
\end{lemma}

Finally, we give the corrected version of \cite[Thm. 5]{Tripathi}, where only the statement of (b) needs to be changed. For completeness we also give the proof of part (a), which still holds despite the new definition of $\mu$. It turns out that the proof of \cite[Thm.~5(b)]{Tripathi} is valid subject to  replacing an incorrect value; see Remark~\ref{rmk:distances}. As done in other instances, for clarity we provide some more justification in the proof of both parts (a) and (b).  Also, we first prove an auxiliary result which will be key in the corrected  proof of part (b).

\begin{lemma}\label{lem:difference}
   Suppose $\ell>k$ and $br>cq$. With reference to Notation~\ref{notation}, Lemma~\ref{lem:11} and Lemma~\ref{lem:rmk8}, write
    $$\mathscr{X}=\{x_i:=r\left(\left\lfloor\tfrac{ui}{r}\right\rfloor+1\right)-ui\mid 0\le i\le \mu\}$$
    and assume $\mu>\lfloor\tfrac{r}{u}\rfloor$. Then  $x_\mu$ is the smallest element in $\mathscr{X}$ that is larger than~$u$. Furthermore, upon ordering the elements in~$\mathscr{X}$ in increasing order, the difference between two consecutive  elements in~$\mathscr{X}$ is either $\widehat{x}:=\min \mathscr{X}$    or    $\widehat{x}+u-x_\mu$.
\end{lemma}

\begin{proof}
  We first show that $x_\mu>u$. To this end, recall that
\[
x_\mu=r\left(\left\lfloor\frac{\mu u}{r}\right\rfloor+1\right)-\mu u=r\left(\left\lfloor\frac{\mu B }{A}\right\rfloor+1\right)- \mu u.
\]
Recall also that $ \frac{u}{r}<\frac{B}{A}$, so from the definition of $\mu$, we have
\begin{align*}
\mu u&=\left\lfloor \frac{\mu u}{r}\right\rfloor r+\delta\\
\mu B&=\left\lfloor \frac{\mu u}{r}\right\rfloor A+\varepsilon
\end{align*}
for some  $0\le \delta<r$ and $0\le \varepsilon<A$, and
\begin{align*}
(\mu +1) u&=\left\lfloor \frac{\mu u}{r}\right\rfloor r+(\delta+u)\\
(\mu +1) B&=\left\lfloor \frac{\mu u}{r}\right\rfloor A+\varepsilon +B=\left( \left\lfloor \frac{\mu u}{r}\right\rfloor +1\right)A+(\varepsilon+B-A)
\end{align*}
with $\delta+u<r$ and $0\le \varepsilon+B-A<A$. Hence, since $u=\left\lfloor \frac{\mu u}{r}\right\rfloor r-\mu u+(\delta+u)$ with $\delta+u<r$, we get that
\[
x_\mu=\left\lfloor \frac{\mu u}{r}\right\rfloor r-\mu u +r>u.
\]

It remains to show that $x_\mu$ is the smallest element in $\mathscr{X}$ that is larger than~$u$. Recall that the expression $r-ui$ is  linear  in $i$, and therefore $(r-ui) \bmod r$ satisfies a corresponding periodic linearity. Hence, if we view the distribution of the elements in $\mathscr{X}$ graphically as points $(i,x_i)$ in $\mathbb{R}^2$, for $i\in\{0,1,\ldots,\mu\}$, we see that the pattern of the distribution is periodic within the intervals $[0,u], [u,2u],\ldots, [(\lfloor \tfrac{r}{u}\rfloor-1) u,\lfloor \tfrac{r}{u}\rfloor u]$ along the vertical axis of our graph, and also correspondingly for the last interval $[\lfloor \tfrac{r}{u}\rfloor u,r]$. Indeed, for every $x_j\in \mathscr{X}$ with $x_j<u$, we have that $u<x_{j-1}=x_j+u<2u$,  $2u<x_{j-2}=x_j+2u<3u$, and so on.

If $m\in\{0,1,\ldots,\mu-1\}$ is such that $x_m=\widehat{x} =\min \mathscr{X}$, this linearity immediately implies that the only distances between consecutive increasing elements in~$\mathscr{X}$ are the constant distance $d_1$ between $x_m$ and the next term in the ordered list of elements in $\mathscr{X}$, and the distance $d_2$ between the last term of the finite arithmetic sequence $x_m$, $x_m+d_1$, $\ldots$, and the term in $\mathscr{X}$ immediately following it. Graphically, one can view the distribution of the elements in $\mathscr{X}$ as points lying on finite-length slopes of the same negative, or positive, gradient, with the first slope ending, respectively beginning, at $x_m$, and the next such slope has its lowest point (in~$\mathscr{X}$) a vertical  distance of $d_2$ from the highest point (in $\mathscr{X}$) of the previous slope. See Figures~\ref{fig:1} and \ref{fig:2} for some examples.

Consider the finite arithmetic sequence $x_m$, $x_m+d_1$, $\ldots$. If the gradient of the slope going through this sequence is negative, then, appealing to the aforementioned periodicity within the intervals $[0,u]$, $[u,2u]$, etc, we deduce that $x_{\mu}$ is the last point (in $\mathscr{X}$) going through the lowest slope in the vertical interval $[u,2u]$. In particular, as mentioned above, for every  $x_\mu\ne x_j\in \mathscr{X}$ with $u<x_j<2u$, we have that $0<x_{j+1}=x_j-u<u$. Therefore, all elements in $\mathscr{X}$ lying on the slope just below the one ending in $x_\mu$, are strictly less than $u$, and these elements are also in one-to-one correspondence with the elements in $\mathscr{X}$ on the lowest slope going through $x_m$. If instead the gradient of the slope going through the finite arithmetic sequence $x_m$, $x_m+d_1$, $\ldots$ is positive, then $x_{\mu}$ is the last point in $\mathscr{X}$ going through the highest slope which begins in the interval $[0,u]$. Similarly, all elements in $\mathscr{X}$ lying on the slope  ending in $x_\mu$, apart from $x_\mu$ itself, are strictly less than $u$, and these elements are also in one-to-one correspondence with the elements in $\mathscr{X}$ on the lowest slope going through $x_m$.  Hence it follows that $x_\mu$ is the smallest element in $\mathscr{X}$ that is greater than $u$.

 For the final statement, with the identification of 0 and $r$,  the periodicity of the slopes yield that one of these distances is the distance between 0 and $x_m = \widehat{x}$, and by the above, the other distance is $x_{m-1}-x_\mu = \widehat{x} + u - x_\mu$, as required.
\end{proof}

\begin{remark}\label{rmk:distances}
    It appears from the proof of~\cite[Thm.~5(b)]{Tripathi} that it was mistakenly assumed that $m=\lfloor \tfrac{r}{u} \rfloor$ (Figures~\ref{fig:1} and \ref{fig:2} above also show that this is not the case) and that $x_\mu=x_{2\lfloor \tfrac{r}{u} \rfloor}$. Hence in~\cite[Proof of Thm.~5(b)]{Tripathi}, the two differences were stated to be $x_m$ and $\lceil \tfrac{r}{u} \rceil u-r$ . The resulting incorrect claim concerning the distances between consecutive increasing elements in~$\mathscr{X}$ is the only oversight in the proof of \cite[Thm.~5(b)]{Tripathi}.
\end{remark}

\begin{theorem}
    [{cf. \cite[Thm.~5]{Tripathi}}]\label{thm:5} 
    Using the above notation, the following hold: 
    \begin{enumerate}
        \item [\textup{(a)}] If $\mu<\lfloor \tfrac{r}{u}\rfloor$, then
        \begin{align*}
        &g(a,b,c)+a=\\
        &\quad\max\left\{b(r-\mu u-1),b(u-1)+c\big(\mu(q+1)+(\lfloor\tfrac{(a-\ell-r)\mu}{r}\rfloor+1)q\big)\right\}+cq\lfloor\tfrac{a-\ell-1}{r}\rfloor.
\end{align*}
Furthermore, if $\Lambda>\Delta$, then 
 \begin{align*}
        &g(a,b,c)+a=\\
        &\quad\max\left\{b(r-\Delta(a-\ell-r)-1),b(a-\ell-r-1)+c\big(\Delta(q+1)+q\big)\right\}+cq
\end{align*}
and if $\Delta'>\Lambda'$, then 
 \begin{align*}
        &g(a,b,c)+a=\max\left\{b(r-1),b\big((a-\ell-1) \bmod r\big) +cq\right\}+cq\lfloor\tfrac{a-\ell-1}{r}\rfloor.
\end{align*}
         \item [\textup{(b)}] 
        If $\mu>\lfloor\tfrac{r}{u}\rfloor$,        recall from Lemma~\ref{lem:11}  that $\mathscr{X}=\{x_i\mid i\in\{0,1,\ldots, \mu\}\}$, where
         \[
         x_i=r\left(\left\lfloor\tfrac{(a-\ell-r)i}{r}\right\rfloor+1\right)-(a-\ell-r)i,
         \]
         and set $$y_i=q\left(\left\lfloor\tfrac{(a-\ell-r)i}{r}\right\rfloor+1\right)+(q+1)i$$ for $i\in\{0,1,\ldots, \mu\}$. Suppose $m\in\{0,1,\ldots, \mu\}$ is such that $x_m:=\min\mathscr{X}$, and let $w\in\{0,1,\ldots, \mu\}$ be maximal such that $x_{w}+x_m\in\mathscr{X}$. Then
         \[
         g(a,b,c)+a=\max\left\{b(x_m-1)+cy_{w},\, b(x_{m-1}-x_{\mu}-1)+cy_{\mu}\right\}+cq\lfloor\tfrac{a-\ell-1}{r}\rfloor.
         \]
    \end{enumerate}
\end{theorem}

\begin{proof}
    (a) If $\mu<\lfloor \tfrac{r}{u}\rfloor$, then from Lemma~\ref{lem:11} we have $\mathscr{X}=\{r-ui\mid  i\in\{0,1,\ldots, \mu\}\}$. So for $i\in\{0,1,\ldots, \mu\}$, 
    $$x_i=r-ui=r\left(\left\lfloor\tfrac{(a-\ell-r)i}{r}\right\rfloor+1\right)-(a-\ell-r)i\in\mathscr{X}
    $$
    and  $\mathbf{m}(bx_i)=cy_i$, where
    \[
    y_i=(q+1)i+q\left(\left\lfloor\tfrac{(a-\ell-r)i}{r}\right\rfloor+1\right)
    \]
    since we need $i$ $\uparrow$-steps and $\left(\left\lfloor\tfrac{(a-\ell-r)i}{r}\right\rfloor+1\right)$ $\downarrow$-steps to arrive at $\mathbf{v}(0,y_i)$ from $\mathbf{v}(x_i,0)$. Note that clearly $\max \{cy_i\mid i\in\{0,1,\ldots, \mu\}\}=cy_{\mu}$.

    Now, if $x<\widehat{x}:=\min \mathscr{X}=x_\mu=r-\mu u$, then $\mathbf{m}(bx)=bx$ by Definition~\ref{def:3} and Lemma~\ref{lem:rmk6}. Next, note that any integer $x\notin\mathscr{X}$, with $\widehat{x}<x<r$, is of the form $x_i+z_i$ with $1\le i\le \mu$  and $0<z_i<u$. Starting from $(x,0)$ and then performing  $i$ $\uparrow$-steps and $(\lfloor\tfrac{(a-\ell-r)i}{r}\rfloor+1)$ $\downarrow$-steps (in the same sequence as for going from  $\mathbf{v}(x_i,0)$ to $\mathbf{v}(0,y_i)$) leads to the local minimum $\mathbf{v}(z_i,y_i)$. Indeed, as $\mathbf{v}(x_i,0)$ and $\mathbf{v}(0,y_i)$ are local minima, shifting their first coordinates by $z_i$ does not change their property of being local minima. In particular, since  $0<z_i<u$ and $x_{i-1}=x_i+u$, we observe that 
    $$
    \mathbf{v}\big((x_i+j(a-\ell)) \bmod a,j\big)<\mathbf{v}\big((x_i+z_i+j(a-\ell))\bmod a,j\big)<\mathbf{v}\big((x_{i-1}+j(a-\ell))\bmod a,j\big)
    $$ 
    for all $0\le j\le y_i$ apart from when $(x_{i-1}+j(a-\ell))\bmod a<u$. From the definition of $x_{i-1}$ and from the definition of the $\mathbf{v}$-values, since $r\not\equiv 0\bmod u$ we see that $(x_{i-1}+j(a-\ell))\bmod a<u$ happens only once, when $(x_{i-1}+j(a-\ell))\bmod a=0$. In this case, $j=y_{i-1}$ and we still have $$\mathbf{v}\big((x_i+y_{i-1}(a-\ell)) \bmod a,y_{i-1}\big)<\mathbf{v}\big((x_i+z_i+y_{i-1}(a-\ell))\bmod a,y_{i-1}\big).
    $$
    Hence our claim is justified.
    
    Now note that $\mathbf{v}(0,y_i)$ is the minimum $\mathbf{v}$-value in the class $[bx_i]$. Then, as reasoned above, upon shifting all the $\mathbf{v}$-values by $bz_i$ yields that $\mathbf{v}(z_i,y_i)$ is the minimum $\mathbf{m}(bx)$ within the congruence class $[bx]$, for $x=x_i+z_i$. Note also that $u\le \widehat{x}=r-\mu u$, as $\mu+1\le \frac{r}{u}$. So
 \[
 \max_{0\le x< u}\mathbf{m}(bx)+cq =b(u-1)+cq,
 \]
and since, by the above discussion, we have $\mathbf{m}(\widehat{x}+u-1)=\mathbf{v}(u-1,y_\mu)$, we see that 
 \begin{align*}
 \max\left\{\max_{0\le x< u}\mathbf{m}(bx)+cq,\, \max_{u\le x< r}\mathbf{m}(bx)\right\}&=\max_{u\le x< r}\mathbf{m}(bx)\\
 &=\begin{cases}
     \mathbf{v}(u-1,y_\mu) & \text{if }u=\widehat{x},\\
     \max \left\{\mathbf{v}(\widehat{x}-1,0),\, \mathbf{v}(u-1,y_{\mu})\right\} & \text{if }u<\widehat{x}.
 \end{cases}
 \end{align*}
 Hence, by  Lemma~\ref{lem:9}  and also by Lemma~\ref{lem3}, we obtain from the above considerations that
\begin{align*}
    g(a,b,c)+a&=\max \left\{\max_{0\le x< u}\mathbf{m}(bx)+cq,\, \max_{u\le x< r}\mathbf{m}(bx)\right\}+cq\left\lfloor\tfrac{a-\ell-1}{r}\right\rfloor\\
    &=\max \left\{\mathbf{v}(\widehat{x}-1,0),\, \mathbf{v}(u-1,y_{\mu})\right\}+cq\left\lfloor\tfrac{a-\ell-1}{r}\right\rfloor\\
    &=\max \left\{b(r-\mu u-1),\, b(u-1)+c\big((q+1)\mu+q(\big\lfloor\tfrac{(a-\ell-r)\mu}{r}\big\rfloor+1)\big)\right\}\\
    &\quad +cq\left\lfloor\tfrac{a-\ell-1}{r}\right\rfloor.
\end{align*}
Hence the first statement. 

For the next statement, note that from the proof of Lemma~\ref{lem:11}, if $\Lambda>\Delta$ or if $\Delta'>\Lambda'$ we have $\Delta=\mu<\tfrac{r}{u}$. When $\Lambda>\Delta$, from the proof of Lemma~\ref{lem:11} we have that  $u=a-\ell-r$. It then follows that $\lfloor \tfrac{a-\ell-1}{r}\rfloor=1$. Indeed, since $u<r$ we have $a-\ell-r<r$ and thus $a-\ell-r+(r-1)<2r-1$, which gives that  $\lfloor \tfrac{a-\ell-1}{r}\rfloor<2$. The claim then follows, as $\lfloor \tfrac{a-\ell-1}{r}\rfloor\ge 1$ from the fact that $r\le a-\ell-1$. Finally, when $\Delta'>\Lambda'$, then $\Delta=0$, and the corresponding result then follows.

\smallskip

(b) Recall that $\mathbf{m}(bx)=bx$ for all $x<\widehat{x}:=\min \mathscr{X}$ by Lemma~\ref{lem:rmk6}, and $\mathbf{m}(bx)=cy$ for $x\in\mathscr{X}$ by Definition~\ref{def:3}. For $x\notin\mathscr{X}$ with $x>\widehat{x}$, as before choose the largest element $x_j\in\mathscr{X}$ such that $x_j<x$ and write $x=x_j+z$ for some $z>0$. Note that such an element $x_j$ exists since  $\mathscr{X}$ is non-empty as $\mu\ne 0$. Recall from the above also that $\mathbf{m}(bx_j)=cy_j$. We next consider the sequence of $\uparrow$-steps and $\downarrow$-steps from $(x_j,0)$ to $(0,y_j)$, and we apply this sequence to $(x,0)$, which brings us to $(z,y_j)$.   Therefore, as argued in part (a), 
\begin{equation}\label{eq:m(x-1)}
\mathbf{m}(bx)=\mathbf{m}(bx_j)+b(x-x_j)=b(x-x_j)+cy_j.
\end{equation}
Hence by Lemma~\ref{lem:9},
\begin{align*}
    \max_{0\le x<a}\mathbf{m}(bx)&=\max \left\{\max_{0< x< u,\, x\in\mathscr{X}}\mathbf{m}(b(x-1))+cq,\, \max_{u\le x< r,\, x\in\mathscr{X}}\mathbf{m}(b(x-1))\right\}+cq\left\lfloor\tfrac{a-\ell-1}{r}\right\rfloor.
\end{align*}
    Indeed, from the above, we see that the maximum is achieved among all $x\notin\mathscr{X}$, and there the maximum is when $z$ is as large as possible, so $x$ is as far away from the nearest $x_j$ as possible.

Recall from Lemma~\ref{lem:difference} that there are exactly two differences that occur between consecutive increasing elements in~$\mathscr{X}$, and we denote these two differences  by $d_1=x_{m-1}-x_\mu$ and $d_2=\widehat{x}$. 
Hence, appealing to \eqref{eq:m(x-1)} and recalling from Lemma~\ref{lem:difference} that $x_\mu>u$, it follows that 
\begin{equation}\label{eq:final}
    \max_{0\le x<a}\mathbf{m}(bx)= \left(\max_{u\le x< r,\, x\in\mathscr{X}}\mathbf{m}(b(x-1))\right)+cq\left\lfloor\tfrac{a-\ell-1}{r}\right\rfloor.
\end{equation}
Again from \eqref{eq:m(x-1)}, to determine the above maximum, it suffices to consider, for $i\in\{1,2\}$,  the largest $j_i\in\{1,\ldots,\mu\}$ such that $x_{j_i}+d_i$ is in~$\mathscr{X}$.  Recall that $w$ is maximal such that $x_w+d_2=x_w+x_m=x_w+\widehat{x}$ is in~$\mathscr{X}$. From Lemma~\ref{lem:difference}, we have that $x_{m-1}=x_\mu +d_1\in\mathscr{X}$.   The result then follows from \eqref{eq:final} and \eqref{eq:m(x-1)}. 
\end{proof}

As promised in the introduction, here is the analogously corrected statement of the parallel result \cite[Thm.~6(b)]{Tripathi} to \cite[Thm.~5(b)]{Tripathi}. We first recall from \cite{Tripathi} that
\[
\overline{q}:=\left\lfloor\frac{a}{\ell}\right\rfloor,\qquad\overline{r}:=a-\overline{q}\ell.  
\]
and
\[
\overline{\mathscr{X}}:=\{x\in\{0,1,\ldots,\ell-\overline{r}\}\mid \mathbf{m}(bx)\equiv  0\bmod c\}.
\]
Further
\[
\overline{A}:=b(\ell-\overline{r})-c(\overline{q}+1),\qquad \overline{B}:=b\overline{r}+c\overline{q},\qquad \overline{u}=\overline{r} \bmod {(\ell-\overline{r})},
\]
and
\[
\overline{\mu}:= \mathbf{\min}\left\{i \in \mathbb{Z}_{\ge 0}\, \bigg\vert \left\lfloor \frac{(i+1)\overline{B}}{\overline{A}}\right\rfloor \neq \left\lfloor \frac{(i+1)\overline{r}}{\ell-\overline{r}}\right\rfloor\right\}.
\]
As for $\mu$, the $\overline{\mu}$ is defined here in an analogously different manner as compared to the definition in~\cite{Tripathi}.
\begin{theorem}
    [{cf. \cite[Thm.~6(b)]{Tripathi}}]\label{thm:6(b)} 
   Suppose $\overline{\mu}>\lfloor\tfrac{\ell-\overline{r}}{\overline{u}}\rfloor$.         With reference to \cite[Lem~12]{Tripathi},  write $\overline{\mathscr{X}}=\{x_i\mid i\in\{0,1,\ldots, \overline{\mu}\}\}$, where
         \[
         x_i=(\ell-\overline{r})\left(\left\lfloor\frac{\overline{r}i}{\ell-\overline{r}}\right\rfloor+1\right)-\overline{r}i,
         \]
         and set $$y_i=(\overline{q}+1)\left(\left\lfloor\frac{\overline{r}i}{\ell-\overline{r}}\right\rfloor+1\right)+\overline{q}i$$ for $i\in\{0,1,\ldots, \overline{\mu}\}$. Suppose $m\in\{0,1,\ldots, \overline{\mu}\}$ is such that $x_m:=\min\overline{\mathscr{X}}$, and let $w\in\{0,1,\ldots, \overline{\mu}\}$ be maximal such that $x_{w}+x_m\in\overline{\mathscr{X}}$. Then
         \[
         g(a,b,c)+a=\max\left\{b(x_m-1)+cy_{w},\, b(x_{m-1}-x_{\bar{\mu}}-1)+cy_{\bar{\mu}}\right\}+c\left((\overline{q}+1)\left\lfloor\frac{\ell-1}{\ell-\overline{r}}\right\rfloor-2\right).
         \]
\end{theorem}

 \begin{figure}[h!]
 \centering
  \includegraphics[width=0.83\textwidth]{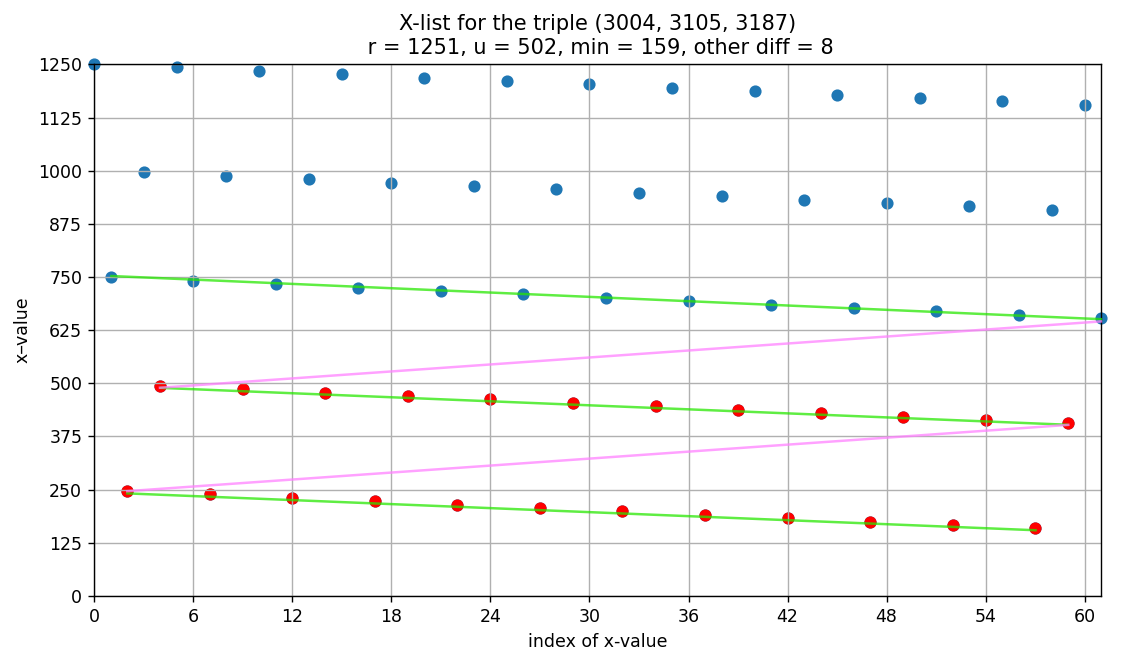}
 \caption{\label{fig:1}An example of a graphical depiction of the values in $\mathscr{X}$ where the slope (depicted by the green line) going through  the lowest points has a negative gradient. The red dots denote the $x$-values less than $u$.}
 \end{figure}
 
 \begin{figure}[h!]
 \centering
  \includegraphics[width=0.9\textwidth]{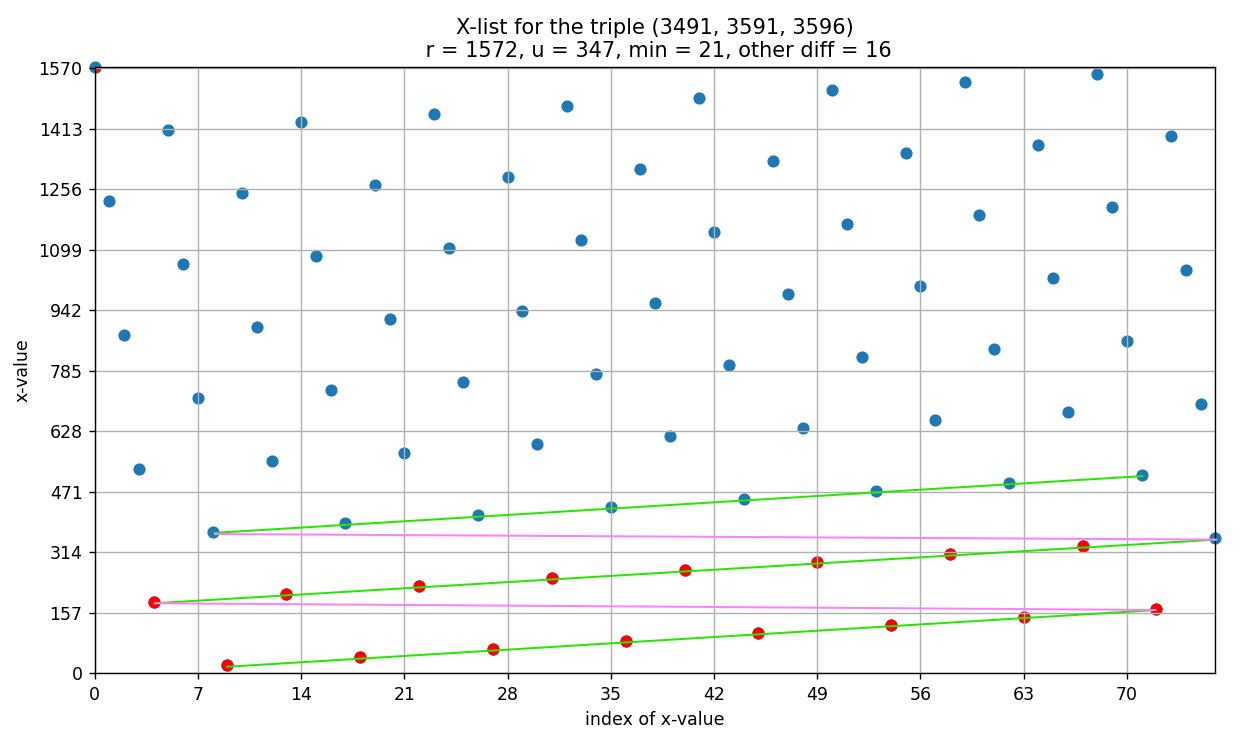}
 \caption{\label{fig:2}An example of a graphical depiction of the values in $\mathscr{X}$ where the slope (see the green lines) going through  the lowest points has a positive gradient. The height of the green and pink lines represent the distances $d_1 = \widehat{x}$ and $d_2 = \widehat{x}+u-x_\mu$ respectively.}
 \end{figure}

%%%%%

\end{document}